\documentclass[journal,twoside,print]{ieeecolor}

\usepackage{generic}
\usepackage{graphicx}

\usepackage{amsfonts}
\usepackage{amsmath}
\usepackage{amssymb}
\usepackage{srcltx}
\usepackage{color}
\usepackage{epsfig}
\usepackage{url}

\newcommand{\RR}{{\mathbb R}}
\newcommand{\NN}{{\mathbb N}}

\newcommand{\cA}{{\mathcal{A}}}

\newcommand{\cE}{{\mathcal{E}}}

\newcommand{\cV}{{\mathcal{V}}}

\newcommand{\cB}{{\mathcal{B}}}

\newcommand{\bx}{{\boldsymbol{x}}}
\newcommand{\by}{{\boldsymbol{y}}}
\newcommand{\bz}{{\boldsymbol{z}}}

\newcommand{\bp}{{\boldsymbol{p}}}

\newcommand{\ba}{{\boldsymbol{a}}}

\newcommand{\bss}{{\boldsymbol{s}}}
\newcommand{\nill}{{\boldsymbol{0}}}
\newcommand{\bzero}{{\boldsymbol{0}}}
\newcommand{\vardot}{\mathord{\,\cdot\,}}

\newcommand{\bigO}{\mathcal O}
\newcommand{\co}{\operatorname{co}}

\newcommand\Tstrut{\rule{0pt}{2.6ex}}         
\newcommand\Bstrut{\rule[-0.9ex]{0pt}{0pt}}   

\newtheorem{theorem}{Theorem}
\newtheorem{lemma}[theorem]{Lemma}
\newtheorem{corollary}[theorem]{Corollary}
\newtheorem{remark}[theorem]{Remark}
\newtheorem{example}[theorem]{Example}
\newtheorem{definition}[theorem]{Definition}
\newtheorem{algorithm}[theorem]{Algorithm}

\usepackage[dvipsnames]{xcolor}
\usepackage{tikz}
\usepackage{tkz-euclide}
\newcommand{\norm}[1]{\left|\left|#1\right|\right|}
\newcommand{\abs}[1]{\left|#1\right|}
\usepackage{rotating}
\newcommand{\m}{\phantom{m}}
\usepackage{multicol}

\usepackage{hyperref}
\hypersetup{hidelinks}

\begin{document}

\title{Stability under dwell time constraints:\\Discretization revisited }

\author{%
Thomas Mejstrik, and 
Vladimir Yu. Protasov
%
\thanks{%
Submitted: 22. January 2024.
This work was supported by the Austrian Science Foundation (FWF)
under Grant P33352-N.
}
\thanks{%
Thomas Mejstrik, University of Vienna;
{\tt\small thomas.mejstrik@gmx.at}
}
\thanks{
Vladimir Yu. Protasov, University of L'Aquila, Italy; 
{\tt\small vladimir.protasov@univaq.it}
}
}

\maketitle

\begin{abstract}
We decide the stability and compute  the Lyapunov exponent 
of continuous-time linear  switching systems with a guaranteed dwell time.  
The main result asserts that the discretization method with step size~$h$ approximates 
the Lyapunov exponent with the precision~$C\,h^2$, where~$C$ is a constant. 
Let us stress that without the dwell time assumption, the approximation rate is known to be linear in~$h$. Moreover, for every system, the constant~$C$ can be explicitly evaluated. In turn, the discretized system can be treated by computing the Markovian joint spectral radius of a certain system on a graph. This gives the value of the Lyapunov exponent 
with a high accuracy. The method is efficient for dimensions up to, approximately, ten;
for positive systems, the dimensions can be much higher, up to several hundreds. 
\end{abstract}


\section{Introduction}

We consider a linear switching system of the form
\begin{equation}\label{eq.main}
\left\{
\begin{array}{l}
\dot \bx (t)  =  A(t)\bx(t), \quad t \in [0, +\infty), \quad A(t)  \in  \cA\\
\bx(0)  = 0
\end{array}
\right. 
\end{equation}
with a positive dwell time restriction. 
This is a 
linear  ODE 
on the vector-function $\bx  :  \RR_+  \to  \RR^d$ 
 with a matrix function $A(t)$ 
taking values from a given finite  set $\cA = \{A_1, \ldots , A_n\}$ called 
 \emph{control set} of matrices (\emph{regimes, modes}). The \emph{control function}, 
 or the  
\emph{switching law}  is an arbitrary piecewise constant function 
$A: \RR_+ \to \cA$ with the lengths of every stationary interval (\emph{switching interval})
at least~$m$, where~$m> 0$ is a given dwell-time constraint. This dwell time assumption has the practical meaning that 
the switches between regimes are not instantaneous but take some positive time. 

For the sake of simplicity we consider only the case 
of finite control sets~$\cA$ and of the same dwell time for all regimes~$A_j$. 
All our results are easily extended to general conditions, see 
Remark~\ref{r.10}.

\subsection{Statement of the problem}

Systems~\eqref{eq.main} regularly arise in engineering applications,  
see, for instance,~\cite{B08, K10, L03, SS00} and references therein. 
One of the main problems is to find or estimate the fastest possible growth rate of trajectories, in particular, to decide about the stability of the system. The stability problem under 
 the dwell time restrictions have been studied in numerous work~\cite{BCM10, BS13, CC17, CCGMS12, CGPS21, SFS15, X15}.

The {\emph Lyapunov exponent} 
$\sigma = \sigma (\cA)$ of a linear switching system  is the infimum of the numbers~$\alpha$ such 
that for every trajectory~$\bx(\vardot)$, we have 
    $\|\bx(t)\|  \le  C e^{\alpha  t}$, $t \ge 0$, 
    for some constant~$C = C(\bx)$. Clearly, if $\sigma < 0$, then the system is asymptotically stable, 
    i.e., all its trajectories tend to zero as~$t\to +\infty$. 
 The converse is also true, although less trivial~\cite{BCM10, MP89}. Thus, the asymptotic stability is equivalent to the inequality~$\sigma < 0$. If, in addition, the system is 
 irreducible, i.e., the matrices from~$\cA$ do not share common nontrivial invariant subspaces, 
 then the inequality~$\sigma \le 0$ is equivalent to the (usual) stability, when all 
 trajectories are bounded. 
 Let us note that if the matrices of the control set~$\cA$ 
 have a common invariant subspace, then the computation of the Lyapunov exponent is reduced 
 to  problems in smaller dimensions by a common matrix factorization~\cite{L03}. Therefore, in what follows we additionally assume that 
 the system is irreducible.

Most of results on the stability of a linear switching systems 
have been obtained without the dwell time assumption. 
Usually it is done either by 
constructing a Lyapunov function \cite{BM99, GC06, GLP17, LM99, MP89},
or by approximating of the trajectories~\cite{PV09, RCS11, SCSS11}.
The latter includes the discretization approach, which is, of course,  the most 
 obvious way to analyse~ODEs. Replacing the derivative $\dot \bx (t)$ by 
 the divided difference $\frac{\bx(t+h) - \bx(t)}{h}$, we get the Euler 
 piecewise-linear approximation. Another way of discretization is  to replace 
 the switching law by a piecewise-constant function with intervals being 
 multiples of a given step size~$h > 0$. Solving the 
 corresponding ODE in each interval we obtain 
 ~$\bx(t+h) = e^{hA_j}\bx(t)$, which gives a piecewise-exponential approximation
 of the trajectory. 
 Both of those methods lead 
  to a discrete time switching system of the form~$\by_{k+1} = B(k)\by_k$, $k\ge 0$. 
 For the Euler discretization, the control set~$\cB$ consists of matrices~$B_j = I + hA_j$, for the 
 second approach, $B_j = e^{hA_j}$. The  Lyapunov exponent of the discrete system
 (we denote it by~$\sigma_h$) can be efficiently computed by evaluating the  
\emph{joint spectral radius}
 of the matrix family~$\cB$~\cite{B88}. 
 The recent progress in the joint spectral radius problem~\cite{GP13, M20}
 allows us to calculate it with a good accuracy or, in most cases, even to find it precisely.  
Thus,  the  Lyapunov exponent $\sigma(\cA)$ of system~\eqref{eq.main}
 can be efficiently computed, provided it is close enough to the Lyapunov exponent~$\sigma_h$ 
 of its discretization.
 The crucial problem is to estimate the precision, i.e., 
 the difference~$|\sigma(\cA) - \sigma_h|$, depending on the step size~$h$.

The discretization method in the stability problem has drawn much attention in the recent 
literature and several estimates for the precision~$|\sigma(\cA) - \sigma_h|$
have been obtained~\cite{GLP17, PJ16, RCS11, SCSS11}.  For both aforementioned methods,
it  is linear in~$h$, i.e.,~$|\sigma(\cA) - \sigma_h| \le Ch$, where, for 
 the constant~$C$, usually only rough upper bounds are known.

\subsection{Overview of the main results}
We establish lower and upper bounds  which localize the Lyapunov exponent of the system~\eqref{eq.main} to an interval  of length at most~$C\,h^2$,
and moreover, show that $C$ can be explicitly found. 
 
 The main result, Theorem~\ref{th.20}, 
 states that the second method of discretization, when~$B_j = e^{hA_j}$, leads to  
 the double inequality~$\sigma_h  \le  \sigma(\cA)  \le  \sigma_h  +   C\,h^2$, where 
 the constant~$C$ is expressed  by means of the so-called  \emph{extremal multinorm}. 
 This multinorm  is 
 constructed simultaneously  with the computation of~$\sigma_h$. 
 
 The estimate from~Theorem~\ref{th.20}  gives a 
 very precise method of 
 computation of the Lyapunov exponent~$\sigma(\cA)$. The numerical results 
in  dimensions $d\le 9$ are given in Section~6. 
In dimension~$d=6$ the  precision is usually between~$0.1$ and $0.15$.  The computation time 
on our PC (5 cores, 3.6~GHz) is about one hour. For positive systems,  
 the method performs much better even in high dimensions.  For $d  \le 200$,
 the precision does not exceed~$0.01$. For $d\le 100$, the computation takes a few seconds.

The use of the new estimate, however,  is complicated by the fact that 
the stability of a discrete system under the dwell time constraint 
cannot be analysed by the traditional scheme. For such systems,  
the Lyapunov function may not exist at all~\cite{CGP18}.
The computation of the Lyapunov exponent 
requires the concept of 
\emph{restricted} or \emph{Markovian} joint spectral radius~\cite{D14, K14, PEDJ16, PMJ17}, which are special cases of the 
recent theory of dynamical system on graphs developed in~\cite{CGP18, P19, PJ15}. 
That is why we need to do preliminary work to
introduce the system on  graphs and the concept of Lyapunov  multinorm.

\begin{remark}\label{r.10}{\em 
Our estimates for the approximation rate do not include the number of matrices, which makes them 
applicable for arbitrary compact control set~$\cA$.  They are also easily 
generalized to mode-dependent constraints on the dwell time, 
when $m$ is a function of the matrix~$A \in \cA$. 
Finally, the results can be extended to mixed (discrete-continuous) systems 
with hybrid control. In particular, when every switching from the regime~$A_i$ to $A_j$
is realized by a given linear operator~$E_{ji}$ that can be different from~$ e^{\,m A_j}$.
}
 \end{remark}

\subsection{The structure of the paper}

In  Section~\ref{sec:pre} we introduce auxiliary facts and notation such as Markovian joint spectral radius, dynamical systems, and Lyapunov multinorms. 
The fundamental theorem 
and corollaries are   
formulated and discussed in Sections~\ref{sec:fundamental},
the proofs are given in Section~\ref{sec:proof}.
Sections~\ref{sec:alg} and~\ref{sec:num} present the algorithm and analyse 
numerical properties. 

Throughout the paper we denote vectors by bold letters,~$I$ is the identity matrix,~$\rho(X)$ is the spectral radius of the matrix~$X$, 
which is the maximum modulus of its eigenvalues.

\section{Preliminary facts. Dynamical systems on graphs}
\label{sec:pre}

The discretization with  step size~$h$ approximates the system~\eqref{eq.main}
with the dwell time parameter~$m> 0$
 by the discrete-time system~$\by_{k+1} = B(k)\by_k$, where 
for each~$k \ge 0$, the matrix~$B(k)$ is either~$e^{hA_j}$ or~$e^{mA_j}$,
$A_j \in \cA$, see Definition~\ref{d.30} below.  The dwell time assumption 
imposes a firm restriction to the switching law: It must be a sequence of
blocks of the form $B(k+N) \cdots  B(k+1)B(k)  = (e^{hA_j})^{N}e^{mA_j}$, where 
$j \in \{1, \ldots , n\}$,  the length $N\ge 0$ depends  on the block, and 
two neighbouring blocks must have different modes~$j$.  
Thus, each block has to begin with~$e^{ mA_j}$ followed by a 
power of~$e^{ hA_j}$, and then switch to the next block with a different~$j$. 
 The general theory of discrete  systems with constraints on switching laws has been developed in~\cite{D14, CGP18, K14, PMJ17}  and then extended to 
general dynamical systems on graphs. We begin with necessary definitions and notation. 

Consider a directed graph~$G$ with vertices~$v_1, \ldots , v_n$
and edges~$\ell_{ji}$, which may include loops~$\ell_{ii}$. Each vertex~$v_i$ is associated to a finite-dimensional linear 
space~$V_i$. To every edge~$\ell_{ji}$, if it exists, we associate 
a linear operator~$E_{ji}: V_i\to V_j$ and a positive number~$h_{ji}$, which is referred to as the \emph{time of action} of~$E_{ji}$. 
\begin{definition}\label{d.10}
The {\emph dynamical system on the graph~$G$} is
the equation~$\bx_{k+1}  = E(k)\bx_k$, 
on the sequence~$\{\bx_k\}_{k\ge0}$, where for each~$k$, the operator~$E(k)$
is chosen from the set~$\{E_{ji}\}_{j=1}^n$ if~$\bx_k \in V_i$.
The point~$\bx_k$ corresponds to the time~$t_k$  and $t_{k+1} = t_k + h_{ji}$.   
\end{definition}

The usual discrete-time linear switching system~$\bx_{k+1} = B(k) \bx_{k}$, 
$B(k) \in \cB = \{B_j\}_{j=1}^n$,  
corresponds to the case when~$G$ is a complete graph,~$V_j$ are all equal to~$\RR^d$, 
every vertex~$v_j$ has $n$ incoming edges~$\ell_{ji}, i=1, \ldots , n$, 
with the operators~$E_{ji}$ equal to the same operator~$B_j$,
and all time intervals~$h_{ji}$ are equal to~$1$. 

Let us have an arbitrary  dynamical system on a graph. We assume that 
each space~$V_j$ is equipped with a certain norm~$\|\vardot\|_{j}$. The collection of those 
norms~$\{\|\vardot\|_j\}_{j=1}^n$ is called a \emph{multinorm}. 
We use the short notation~$\|\bx\|$ meaning that~$ \|\bx\|  = \|\bx\|_j$
for $\bx \in V_j$.

Every trajectory~$\{\bx_k\}_{k=0}^{\infty}$ of the system on~$G$ 
corresponds to an infinite path~$v_{i_0} \to v_{i_1} \to \cdots $, where  
$\bx_0 \in V_{i_0}$ is a starting point and~$\bx_{k+1} = E_{i_{k+1} i_k}\bx_k$, $k \ge 0$. 
Thus, $\bx_{k} \in V_{i_k}$ for every~$k$. 
Every point~$\bx_k$ corresponds to the 
time~$t_k = \sum_{s=1}^{k}h_{i_{s} i_{s-1}}$, which is the total time 
of the way from~$v_{i_0}$ to~$v_{i_k}$.

The \emph{(Markovian) joint spectral radius} $\hat \rho  = \hat \rho(\cA)$ of the system 
is 
\begin{equation*}
\hat \rho  =  \lim\limits_{k\to \infty}  \max\limits_{\{\bx_s\}_{s=0}^k}  \|\bx_k\|^{1/t_k},
\end{equation*}
where the maximum is computed over all trajectories~$\{\bx_s\}_{s\ge 0}$ with~$\|\bx_0\| =1$. 
Thus, for every trajectory, we have~$\|\bx_k\|  \le  C \hat \rho^{t_k}$, $k\in \NN$.
The joint spectral radius is the rate of the fastest  growth of trajectories. 
       See~\cite{CGP18} for the correctness of the definition and for basic properties of this 
notion. 

Let us now consider an arbitrary cycle of the graph~$G$: $v_{i_0}\to v_{i_1} \to \cdots \to v_{i_k} = v_{i_0}$. Denote by $\Pi$ the product of operators~$E_{i_ki_{k-1}} \cdots  E_{i_1i_0}$
along the cycle. We have~$\Pi\bx_0 = \bx_k$. For the spectral  
radius $\rho(\Pi)$, which is the maximal modulus of eigenvalues of~$\Pi$, we 
have~${[\rho(\Pi)]^{ 1/t_k}  \le  \hat \rho}$.    
Indeed, the left-hand side is the rate of growth of a periodic trajectory 
going along that cycle, and it does not exceed the maximal rate of growth~$\hat \rho$
over all trajectories. 
On the other hand, there are cycles for which the left-hand side is arbitrarily close 
to~$\hat \rho$~\cite{CGP18, K14}. 

\begin{definition}\label{d.20}
A multinorm is called \emph{ extremal} for a system on a graph 
if for each~$i, j \in \{1, \ldots , n\}$
and for every~$\bx \in V_i$, we have $\|E_{ji} \bx\|_j  \le  \hat \rho^{ h_{ji}} \|\bx\|_i$, 
provided the edge~$\ell_{ji}$ exists. 
\end{definition} 
For the extremal multinorm, for every trajectory, we have 
$\|\bx_k\|  \le  \hat \rho^{t_k}\|\bx_0\|$, $k \in \NN$.

The extremal multinorm always exists, provided the system is 
irreducible. The reducibility means the existence 
of subspaces~$V_j' \subset V_j$, $i =1, \ldots , n$, where 
at least one subspace is nontrivial and at least one inclusion is strict, such that 
$E_{ji}V_i' \subset V_j'$, whenever the edge~$\ell_{ji}$ exists. 


\noindent \textbf{Theorem A}~\cite{CGP18}
\emph{An irreducible dynamical system on a graph 
possesses an extremal multinorm}. 

Theorem A implies, in particular, that for every irreducible system, we have~$\hat \rho> 0$. 
Otherwise, $E_{ji}\bx  = 0$ for all~$\bx\in V_i$, i.e.,  all~$E_{ji}$ are 
equal to zero,  in which case the system is 
clearly reducible. Therefore, one can always 
normalize the operators as $\tilde E_{ji}  = \hat  \rho^{ -h_{ji}}E_{ji}$, after which  the 
new dynamical system (on the same graph and with the same time intervals) 
has the joint spectral radius equal to one. Moreover, it possesses the same extremal norm, 
for which~$\|\tilde E_{ji}\bx\|_j \le \|\bx\|_i$, hence,  
for every trajectory~$\tilde \bx_k$, the sequence~$\|\tilde \bx_k\|$ is non-increasing in~$k$.  

Now we are going to define the discretization of the dwell time constrained system~\eqref{eq.main} and present it as  a system on a suitable graph. Then we apply Theorem~A
and use an extremal multinorm to estimate the approximation rate. 

\begin{definition}\label{d.30}
  The \emph{$h$-discretization of the system~$\cA$   with  step size~$h$} is a system~$\cA_h$ 
on a complete graph with~$n$ vertices, in which~$V_j = \RR^d$, $E_{jj} = e^{\,hA_j}$,
$h_{jj} = h$
for all~$j = 1, \ldots , n$, and $E_{ji} = e^{\,mA_j}$, $h_{ji} = m$  for all 
pairs~$(i,j)$, $i\ne j$.  
\end{definition}
Thus, the $h$-discretization~$\cA_h$ is a dynamical system on a complete graph, 
where every vertex~$v_j$ has $n-1$ incoming edges from all other vertices, all associated to the operator~$e^{\,mA_j}$
with the time of action~$m$, and also has a loop associated to~$e^{\,hA_j}$ with time
of action~$h$. 
Respectively, a multinorm~$\|\vardot \| = \{\|\vardot \|_j\}_{j=1}^n$ can now be interpreted as 
a collection of norms in~$\RR^d$, each associated to the corresponding regime~$A_j$.
See Figure~\ref{fig:hdiscretization} for an example of a system~$\cA$ with three matrices. 

The $h$-discretization can also be presented as a discrete-time linear switching system in $\RR^d$: 
$\bx_{k+1} = B(k)\bx_k, \allowbreak k \ge 0$, where the sequence $B(k)$ has values from the control 
set~$\{e^{\,mA_j},\allowbreak e^{\,hA_j}\}_{j=1}^n$   with the following  restriction:

Every element~$B(k)$  equal to $e^{\,mA_j}$ or~$e^{\,hA_j}$ can be followed 
by either~$e^{hA_j}$ or $e^{mA_s}$, $s\ne j$.

The  operators~$e^{\,mA_j}, e^{\,hA_j}$ act during the time intervals of lengths~$m$ and $h$ respectively. We denote by~$\hat \rho(\cA_h)$ the (Markovian) joint spectral radius of the 
system~$\cA_h$. By Theorem~A, if the family~$\cA$ is irreducible, then 
for every~$h$ its~\mbox{$h$-discretization} possesses an extremal norm. Note that 
this norm can be different for different~$h$.

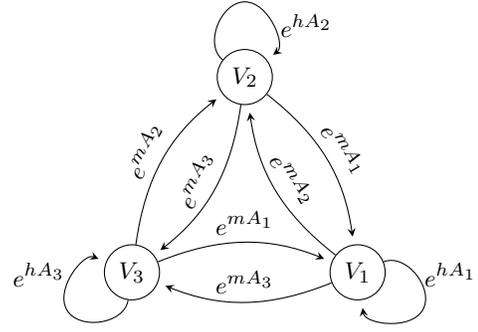
\begin{figure}
\centering\small
\begin{tikzpicture}
    \tkzDefPoint(0.,0.){V1}
    \tkzDefPoint(1.,0.){V2}
    \tkzDefPoint(0.5,0.860){V3}
    \def \d {22};
    \def \sa {3};
    \def \sb {0};
    \def \sc {3};
    \draw (4,4) node[draw, circle] (V1) [] {$V_1$}
      ++(120:3) node[draw, circle] (V2) [] {$V_2$}
      ++(240:3) node[draw, circle] (V3) [] {$V_3$};

    \draw[<-, > = stealth, shorten < = {\sa}, shorten > = {\sb}] (V1) to [out= {120-\d}, in={-60+\d}] node[sloped, midway, above] {$e^{mA_1}$} (V2);
    \draw[->, > = stealth, shorten < = {\sb}, shorten > = {\sa}] (V1) to [out= {120+\d}, in={-60-\d}] node[sloped, midway, above] {$e^{mA_2}$} (V2);
    \draw[<-, > = stealth, shorten < = {\sa}, shorten > = {\sb}] (V1) to [out= {180-\d}, in={  0+\d}] node[sloped, midway, above] {$e^{mA_1}$} (V3);
    \draw[->, > = stealth, shorten < = {\sb}, shorten > = {\sa}] (V1) to [out= {180+\d}, in={  0-\d}] node[sloped, midway, above] {$e^{mA_3}$} (V3);
    \draw[<-, > = stealth, shorten < = {\sa}, shorten > = {\sb}] (V2) to [out= {240-\d}, in={ 60+\d}] node[sloped, midway, above] {$e^{mA_2}$} (V3);
    \draw[->, > = stealth, shorten < = {\sb}, shorten > = {\sa}] (V2) to [out= {240+\d}, in={ 60-\d}] node[sloped, midway, above] {$e^{mA_3}$} (V3);
    
    \draw[<-, > = stealth, shorten < = {\sc}] (V1) to [out= {-60-\d}, in= 240] node[below] {          } ++(-30:1) to [out= 60, in={  0+\d}] node[right] {$e^{hA_1}$}(V1);
    \draw[<-, > = stealth, shorten < = {\sc}] (V2) to [out= { 60-\d}, in=   0] node[right] {$e^{hA_2}$} ++( 90:1) to [out=180, in={120+\d}] (V2);
    \draw[<-, > = stealth, shorten < = {\sc}] (V3) to [out= {180-\d}, in= 120] node[left]  {$e^{hA_3}$} ++(210:1) to [out=-60, in={240+\d}] (V3);
  
\end{tikzpicture}
\vspace*{-2mm}
\caption[]{The dynamical system on graph, $n=3$.}
\label{fig:hdiscretization}
\end{figure}

\section{The fundamental theorem}
\label{sec:fundamental}
Now we are formulating the main result. We consider a linear switching system~$\cA$ given by~\eqref{eq.main}
and its $h$-dis\-cret\-iza\-tion~$\cA_h$. The joint spectral radius of~$\cA_h$ is denoted by~$\hat \rho(\cA_h)$. 
\begin{theorem}\label{th.20}
Let~$\cA$ be an irreducible  continuous-time linear switching system with the dwell time
constraint~$m$.  
Then for  every discretization step $h > 0$, 
 we have  
 \begin{equation}\label{eq.40}
\sigma_h 
\le
\sigma(\cA)
\le
\sigma_h - \frac{1}{m} \ln  \left(1 - \frac{\|(\cA - \sigma_h I)^2\|}{8} h^2 \right), 
 \end{equation} 
where $\sigma_h  = \ln  \hat \rho(\cA_h)$, 
$\|\vardot \| = 
\{\|\vardot \|_j\}_{j=1}^n$ is an extremal multinorm of~$\cA_h$, and 
\begin{equation*}
\|(\cA - \sigma_h I)^2\|  = \max_{j=1, \ldots , n} \|(A_j - \sigma_h I)^2\|_j . 
\end{equation*}
\end{theorem}
\begin{remark}\label{r.40}{\em 
The estimate~\eqref{eq.40} uses two values: the joint spectral radius of 
the $h$-discretization $\hat \rho(\cA_h)$ and the operator norms of~$(A_j - \sigma_h I)^2$
in the $j$th component of the  extremal multinorm~$\{\|\vardot\|_j\}$ for~$\cA_h$. 
They are both  found by the invariant polytope algorithm~\cite{CGP18}.  We  
give a brief description in Section~5. 
}
\end{remark}  

Writing the Taylor expansion of the logarithm up  to the second order, we obtain 
the following
\begin{corollary}\label{c.10}
Under the assumptions of~Theorem~\ref{th.20}, we have 
 \begin{equation}\label{eq.50}
\sigma_h \le \sigma(\cA) \le
\sigma_h  + 
 \frac{\|(\cA - \sigma_h I)^2\|}{8m} h^2  +  \bigO(h^4), \quad \text{as $h\to 0$}. 
 \end{equation} 
\end{corollary}

\begin{remark}{\em 
If $h=m$, then we have basically the discretization of an unrestricted system, and~\eqref{eq.50} becomes
\begin{equation*}
\sigma_h \le  \sigma(\cA) \le \sigma_h  + 
 \frac{\|(\cA - \sigma_h I)^2\|}{8} h  +  \bigO(h^3), \quad \text{as $h\to 0$},
\end{equation*}
which again reveals the linear dependence on the discretization step $h$ for unrestricted systems.
}
\end{remark}

\begin{remark}\label{r.45}{\em 
If $\norm{\vardot}$ is an approximation 
of an extremal multinorm of~$\cA_h$ up to a factor of $1+\varepsilon$ 
i.e.\ (in the notation of Definition~\ref{d.20})
\begin{equation*}
\hat \rho^{h_{ji}} \|\bx\|_i
\le
\|E_{ji} \bx\|_j
\le 
(1+\varepsilon) \hat \rho^{h_{ji}} \|\bx\|_i,
\end{equation*}
then~\eqref{eq.50} becomes
\begin{equation*}
\sigma_h^- 
\le
\sigma(\cA)
\le
\sigma_h^+  - \frac{1}{m} \ln  \left(1 - \frac{\|(\cA - \sigma_h^- I)^2\|}{8} h^2  \right), 
\end{equation*}
where
$\sigma_h^-
=
\ln \hat \rho(\cA_h)$,
and $\sigma_h^+
=
\ln (1+\varepsilon)\hat\rho(\cA_h)$.
}\end{remark}

The quadratic rate of approximation in formula~\eqref{eq.50}
is quite unexpected since the trajectories of the continuous-time 
system are approximated by the trajectories of its~$h$-discretization 
only with the  linear rate. Nevertheless, the approximation of the 
Lyapunov exponent is quadratic.
\begin{remark}\label{r.50}{\em 
The performance of the estimate~\eqref{eq.40}
can be spoiled in two cases: either the dwell time~$m$ is too small,
or the  operator norm of~$(A_j- \sigma_h I)^2$ is too large. 
Note also that~\eqref{eq.40} is an a posteriori estimate since the operator norm 
depends on~$h$ and is not known in advance.
}\end{remark}

Theorem~\ref{th.20} yields the following stability conditions for system~\eqref{eq.main}: 

\begin{corollary}\label{c.20}
If an $h$-discretization~$\cA_h$ is unstable, then so is the 
system~$\cA$.   If~$\cA_h$ is stable
and~$\hat \rho \le \bigl(1 - \frac{h^2\|(\cA - (\ln \rho)I)^2\|}{8} \bigr)^{1/m}$, 
then~$\cA$ is stable. 
\end{corollary}
\begin{proof}
The first inequality in~\eqref{eq.40} implies that if~$\hat \rho > 1$
and hence~$\sigma_h > 0$, then~$\sigma (\cA) > 0$. The converse is established similarly.
\end{proof}

\section{Proofs of the main results}
\label{sec:proof}
The proof of Theorem~\ref{th.20} is based on a simple geometrical argument.  
If a curve connects 
two ends of a segment of length~$h$, then the distance from 
this curve to the segment does not exceed $h^2$ multiplied by the curvature 
and by a certain constant. The nontrivial moment is 
that we need this property in an arbitrary norm in~$\RR^d$ and need to evaluate 
the constant depending on this norm. Then we apply this fact to each component~$\|\vardot \|_j$
of  the extremal norm of the system~$\cA_h$ and estimate the growth of trajectory 
of the system~$\cA$. 

\begin{lemma}\label{l.30}
Let~$\|\vardot\|$ be an arbitrary norm in~$\RR^d$ and 
$\bx: [0,h]\to \RR^d$ be a $C^2$-curve.  
Then, for every~$\tau \in [0,h]$, 
the distance from the point~$\bx(\tau)$ to the segment~$[\bx(0),\bx(h)]$
does not exceed~$\frac{h^2}{8} \|\ddot \bx\|_{C[0, h]}$. 
\end{lemma}
\begin{proof}
Denote by~$\|\cdot\|^*$ the dual norm in~$\RR^d$, 
thus~$\|\by\|^* = \sup_{\|\bx\| = 1}(\by, \bx)$. Let~$\bx(0) = \nill$ and~$\bx(h) = \ba$. 
Let $\bx(\xi)$ be the most distant point of the arc~$\{\bx(\tau), \tau\in [0,h]\}$
to the segment~$[\nill,\ba]$ and let this maximal distance be equal to~$r$. 
Suppose  $\by$ is the closest to~$\bx(\xi)$
point of that segment and denote~$\bss = \bx(\xi) -  \by$. Thus, $\|\bss\| = r$. 
The segment~$[\nill,\ba]$ does not intersect the 
interior of the ball of radius~$r$ centred at~$\bx(\xi)$. 
Therefore, by the convex separation theorem, there exists 
a linear functional~$\bp \in \RR^d$, $\|\bp\|^*=1$,  which is non-positive 
on that segment, non-negative on the ball, and such that 
$(\bp, \bss) = \|\bss\|$. Since the point $\by$ belongs to the ball, it follows that 
$(\bp, \by) = 0$.  We have  $(\bp, \bss) = 
(\bp, \bx(\xi)) -  (\bp, \by)  = (\bp, \bx(\xi))$, therefore, 
$(\bp, \bx(\xi))  = r$, see Figure~2. 

Defining the 
function~$f(t)  = \bigl( \bp, \bx(t)\bigr)$ we obtain
\begin{equation}\label{eq.r}
f(\xi)  - f(0)  = \bigl( \bp ,  \bx(\xi) - \bx(0)\bigr)   = 
  \bigl( \bp, \bx(\xi) \bigr) = r.
\end{equation}
Without loss of generality we assume that~$\xi \le \frac{1}{2} h $, otherwise one can interchange 
the ends of the segment~$[0,h]$. The Taylor expansion 
of~$f$ at the point~$\xi$ gives~$f(t) = f(\xi)  +  f'(\xi)(t-\xi) +  
\frac{1}{2}  f'(\eta) (t-\xi)^2$, where $\eta \in [t, \xi]$. 
The maximum of~$f(t)$  is attained at~$t=\xi$, 
hence,~$f'(\xi) = 0$. For~$t = 0$, this yields  
$f(0) = f(\xi) +  \frac{1}{2}  f''(\eta) \xi^2$. 
Combining with~\eqref{eq.r}, we obtain 
\begin{equation*}
   r
   =
   - \frac{1}{2} f''(\eta) \xi^2
   =
   \frac{1}{2} \bigl(\bp ,  -\ddot \bx(\eta)\bigr) \xi^2
   \le 
   \frac{1}{2} \|\ddot \bx\| \xi^2
   \le  
   \frac{h^2}{8} \|\ddot \bx\|
   .
\end{equation*}
which completes the proof. 
\end{proof}

\begin{figure}
\centering\small
\begin{tikzpicture}
\tkzDefPoint(0,0){Zero}
\tkzDefPoint(2,1){xxi}
\tkzDefPoint(5,0){a}
\tkzDefPoint(2,0){y}
\draw[ultra thick] (Zero)  to[out=90,in=180] (xxi)  to[out=0,in=135] (a);
\draw (xxi) -- (y) node [midway, left] {$s$};
\draw[ultra thick, gray] plot coordinates {(xxi) (y)};
\draw[] (Zero) -- (a);
\draw[](xxi) circle (1);
\tkzLabelPoint[left, below](Zero){$\bx(0)=\bzero$}
\node at (Zero)[circle,fill,inner sep=1.5pt]{};
\tkzLabelPoint[below](y){$\by$}
\node at (y)[circle,fill,inner sep=1.5pt]{};
\tkzLabelPoint[above](xxi){$\bx(\xi)$}
\node at (xxi)[circle,fill,inner sep=1.5pt]{};
\tkzLabelPoint[below](a){$\bx(h)=\ba$}
\node at (a)[circle,fill,inner sep=1.5pt]{};
\draw (2.5,0) -- (2.75,-.5) node [right] {$\bp\leq0$};
\node at (2.5,0)[circle,fill,inner sep=.75pt]{};
\draw (1.5,1.5) -- (1,2) node [above] {$\bp\geq0$};
\node at (1.5,1.5)[circle,fill,inner sep=.75pt]{};
\draw[dotted ] (xxi) -- (2.8660,1.5);
\draw[->,dotted ] (3.7321,2) -- (2.8660,1.5);
\node at (4.4,1.7) {$r \leq \frac{h^2}{8}\|\ddot\bx\|$};
\end{tikzpicture}
\caption{Construction in proof of Lemma~\ref{l.30}.}
\label{fig:30}
\end{figure}
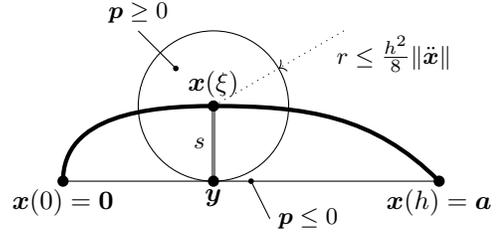

\begin{theorem}\label{th.10}
Let~$\bx(t)$ be a solution of the differential equation~$\dot \bx = A \bx$ with a constant 
$d\times d$ matrix~$A$, $\|\vardot\|$ be an arbitrary norm in~$\RR^d$, 
and $h \in \bigl( 0 ,  \sqrt{8/\|A^2\|} \bigr)$ be a number. 
Then for every~$\tau \in [0,h]$, we have
 \begin{equation}\label{eq.30}
 \bigl\|\bx(\tau) \bigr\|
 \le
 \frac{1}{1 -  \frac{h^2}{8} \|A^2\| }
  \max \Bigl\{ \|\bx(0)\| ,  \|\bx(h)\|   \Bigr\} . 
 \end{equation}
\end{theorem}
\begin{proof}
It suffices to consider the vector~$\bx(\tau)$
with the maximal norm over all~$\tau \in [0,h]$. 
Since  $\ddot \bx = A\dot \bx  = A^2\bx$, it follows that
$\|\ddot \bx\|_{C[0,h]}  = \max_{t\in [0,h]}\|A^2\bx(t)\|  \le  
\|A^2\|\cdot \|\bx(\tau)\|$. Hence, by Lemma~\ref{l.30}, 
the distance from the point~$\bx(\tau)$ to the closest point~$\by$ of the segment~$[\bx(0), \bx(h)]$
does not exceed~$\frac{h^2}{8} \|A^2\|\cdot \|\bx(\tau)\|$. 
On the other hand, this distance is not less than 
$\|\bx(\tau)\| -  \|\by\|$.
It remains to note that $\|\by\|  \le  \max \{\|\bx(0)\|, \|\bx(h)\|\}$, 
which follows from the convexity of the norm. 
Thus,
\begin{equation*}
\frac{h^2}{8} \|A^2\|\cdot \|\bx(\tau)\| \ge 
\|\bx(\tau)\| -   \max \Bigl\{\|\bx(0)\| ,  \|\bx(h)\| \Bigr\} .  
\end{equation*}
Expressing~$\|\bx(\tau)\|$ we arrive at~\eqref{eq.30}.
\end{proof}

\begin{proof}[Proof of Theorem~\ref{th.20}]
The lower bound follows trivially. 
 To prove the upper bound, 
we first assume that~$\hat \rho(\cA_h) = 1$.
By Theorem~A, the system~$\cA_h$ possesses an extremal  multinorm~$\|\vardot \| = \{\|\vardot \|_j\}_{j=1}^{n}$. 
For each~$j\le n$, we denote~$\alpha_j  = - \frac{1}{m} \ln    \left(1 -  \frac{\|A_j^2\|_j}{8} h^2\right) $. 
  Let us show that 
for every~$i\ne j$ and for every point~$\bz_0 \in V_i$, the 
 trajectory~$\bz(t)$ generated in~$V_j$ by the ODE $\dot \bz = A_j \bz$  and starting at~$\bz_0$
 possesses the property:~$\|\bz(t)\|_j  \le  e^{\alpha_jt} \|\bz_0\|_j$ for all $t \ge m$. 
Since the multinorm is extremal, it follows that for every integer~$k\ge 0$, one has
$
\bigl\| \bz(m+kh)\bigr\|_j
=
\bigl\|e^{mA_j} \bigl( e^{hA_j} \bigr)^k  \bz_0 \bigr\|_j 
\le
\|\bz_0\|_j
$. 
  Let~$k$ be the maximal integer such that $ t  \ge  m +  kh$.
  Thus, $t=m+kh+\tau$, $\tau \in [0,h)$. Consider the arc of the 
  trajectory~$\bz(\vardot)$ on the time interval~$[m+kh, m+(k+1)h]$
  and denote~$\bx(0) = \bz(m+kh)$, $\bx(h)  = \bz (m+kh+h)$. 
Observe that both~$\|\bx(0)\|$ and $\|\bx(h)\|$ do not exceed~$\|\bz_0\|$.    
  Applying Theorem~\ref{th.10}
   we obtain
\begin{align*}
\|\bz(t)\|_j & = \|\bx(\tau)\|_j
 \le 
\frac{1}{1 -  \frac{h^2}{8} \|A_j^2\|_j } \|\bz_0\|_j \\
& = 
e^{\alpha_j m} \|\bz_0\|_j  \le 
e^{\alpha_j t} \|\bz_0\|_j.
\end{align*}

Using the multinorm notation and setting $\alpha  = \max_{j=1}^n \alpha_j$, 
we conclude that for every trajectory~$\bz(\vardot)$ generated by one regime it holds that  
$\|\bz(t)\|  \le   e^{\alpha  t} \|\bz(0)\|$ for all~$t \ge m$. 

Consider now an arbitrary trajectory~$\bx(\vardot)$ of the system~$\cA$. 
If it does not have switches, then it is generated by some regime~$A_j$, 
and the proof follows immediately from the inequality above. Let it have the switching points $t_0 < t_1 < \ldots $.  This set can be infinite or finite.
Applying the inequality above to arbitrary~$ t \in [t_k, t_{k+1}]$
and denoting~$\bz(0) = \bx(t_k)$, we obtain~$\|\bx(t)\|   \le  
e^{\alpha (t - t_k)}\|\bx(t_k)\|$. Now take a time interval 
$[0,T]$ and let $t_N$ be the largest switching point on it. 
Applying our inequality successively for all switching intervals, we get 
\begin{align*}
\|\bx(T)\|
& \le
e^{\alpha(t-t_N)} 
\prod_{k=0}^{N-1} e^{\alpha (t_{k+1} - t_k)}  \|\bx(t_0)\|  \\
& \le 
e^{\alpha (T - t_0)}  \|\bx(t_0)\| . 
\end{align*}
Thus, for every trajectory~$\bx(\vardot)$, we have 
 $\|\bx(T)\|  =  \bigO( e^{\alpha t})$ as $T \to \infty$. 
 Let us remember that the norm~$\|\vardot\|$ can be different on different switching intervals. 
 Nevertheless, they all belong to the finite set of norms~$\{\|\vardot\|_j\}_{j=1}^n$
 which are all equivalent. Therefore, 
 $\sigma(\cA) \le \alpha$.  
 
 This concludes the proof for the case~$\hat \rho = 1$. 
 The general case follows from  this one by normalization: 
 We replace the family~$\cA$ by $\tilde \cA  = \cA - \sigma_h I$. 
 Then $\hat \rho(e^{h\tilde \cA}) = 1$ and $\sigma(\tilde \cA) = \sigma(\cA) - \sigma_h$. 
 Finally, applying  the theorem  for the system~$\tilde \cA$
 and substituting to~\eqref{eq.40}, we complete the proof. 
\end{proof}

Now, to put  Theorem~\ref{th.20} into practise, we need to 
compute the value~$\hat \rho (\cA_h)$  and construct an extremal multinorm. 
We will do it in Section~5 for a general system on a graph.

\section{Computing the joint spectral radius and an extremal multinorm for a system on a graph}
\label{sec:alg}

Theorem~\ref{th.20} gives a recipe to compute the Lyapunov exponent 
of a linear switching system~\eqref{eq.main} with sufficiently high accuracy.
Due to the quadratic dependence of~$h$ in~\eqref{eq.50}
one can make the distance between the upper and lower bounds small by choosing an
appropriate discretization step~$h$.
This plan requires solving two problems:
(i) Compute the value of~$\hat \rho(\cA_h)$.
(ii) Construct an extremal multinorm for~$\cA_h$. 
Both are solved simultaneously by the invariant polytope algorithm \emph{(ipa)} derived in~\cite{GP13}.
The ipa computes the joint spectral radius of several matrices by constructing an extremal norm.
In~\cite{CGP18} the invariant polytope algorithm was extended to
discrete time systems with restrictions and to
systems of graphs.
We present the main idea of the algorithm in this section;
for details see~\cite{GP13,M20,CGP18,KP23}.
The reference implementation of the algorithm can be found at~%
\href{https://gitlab.com/tommsch/ttoolboxes}{\emph{gitlab.com/tommsch/ttoolboxes}}.

The (Markovian) invariant polytope algorithm~(\emph{ipa}) finds the joint spectral radius~$\hat \rho (\cA_h)$ and an 
extremal multinorm. It consists of two steps. 
 
 \paragraph{Step 1.}
 We fix a number~$N$ (not very large) and exhaust all cycles
 $v_{i_0}\to v_{i_1} \to \cdots \to v_{i_L}= v_{i_0}$ of~$G$ of length at most~$N$. To each cycle we denote by~$\Pi  = E_{i_{L}i_{L-1}} \cdots E_{i_1i_0}$ the product of the linear operators 
 associated to 
 its edges and by~$T  = \sum_{k=1}^{L}h_{i_{k}i_{k-1}}$
 its total time. Recall that $E_{jj}  = 
 e^{hA_j}$, $h_{jj} = h$ for all~$j$ and $E_{ji}  = e^{mA_j}$ , $h_{ji} = m$  for
all pairs $(i, j)$, $i \ne j$.

 We choose a cycle with the 
 biggest value~$r = \rho(\Pi)^{1/T}$ and call it \emph{leading cycle}
 and respectively~\emph{leading product}~$\Pi$. We set~$\tilde A_j  = A_j -  (\ln r) I$, $j=1, \ldots , n$.
 For the new system~$\tilde \cA_h$, we have 
 $\rho(\tilde \Pi)  = \rho( \tilde E_{i_{L}i_{L-1}} \cdots \tilde E_{i_1i_0}) = 1$. 

In many cases to find the leading cycle
we can avoid the exhaustion and use the auxiliary Algorithm~\ref{alg.10}. It is an adaptation of the modified Gripenberg algorithm from~\cite{M20} to our setting. Some of its numerical properties are assessed in Appendix~A. 
 
For the sake of simplicity of exposition,  we  assume 
that the leading eigenvector of~$\tilde \Pi$ is positive and thus equal to one. The general case is considered similarly, see~\cite{CGP18, MP23}.   
 
Denoting by $\bx_0$ the  leading eigenvector of~$\tilde \Pi$, it follows 
that~$\bx_0 \to \bx_1 \to \cdots \to 
 \bx_{L-1}\to \bx_0$ is the periodic trajectory corresponding to that cycle. 
 
\paragraph{Step 2.}
We try to prove that actually~$\hat \rho( \cA_h) = r$ and, if so, to find an extremal norm. 
 
For every~$j$, we denote by~$\cV_j^{(0)}$
the set of points~$\bx(s)$, $s = 0, \ldots , L-1$ that belong to~$V_j$. 
If there are no such points, then~$\cV_j^{(0)} = \emptyset$. 
Suppose after $k$ iterations we have finite sets~$\cV_j^{(k)} \subset \allowbreak V_j$, $j = 1, \ldots , n$.
Denote by~$\co_s{\cV_j^{(j)}}$
the convex hull of the set~$\cV_j^{(k)} \cup (-\cV_j^{(k)})$.
Now for every $j$, we add to~$\cV_j^{(k)}$ all points of the sets 
$e^{m\tilde A_j}\cV_s^{(k)}$ for all~$ s \ne j$
and of the set $e^{\,h\tilde A_j}\cV_j^{(k)}$.
We add only those points that do not belong to~$\co_s{\cV_j^{(j)}}$,
the others are redundant and we discard them.
This way we obtain the sets~$\cV_j^{(k+1)}$, $j =1, \ldots , n$.  
We do this until~$\cV_j^{(k+1)} = \cV_j^{(k)}$ for all~$j$, in which case 
the algorithm terminates.  We conclude that~$\hat \rho(\cA_h) = \rho(\Pi)^{1/T}$ and the Minkowski norms 
of the polytopes~$\{\co_s{\cV_j^{(j)}}\}_{j=1}^n$ form an extremal multinorm for~$\cA_h$.

\begin{remark}{\em
To find the Lyapunov exponent~$\sigma (\cA)$, we need to compute the operator norms $\|A_j - \sigma_h I\|_j$.
Since the unit balls of the norms $\norm{\vardot}_j$ are polytopes,
this can be done efficiently by solving an LP problem~\cite{GP13}.

To achieve a good precision one needs to choose an appropriate step size $h$ which, however,
cannot be too small. Otherwise, the matrices~$e^{hA_j}$ will be close to the
identity complicating the computation of the joint spectral radius~\cite{GLP17}.
In particular, the length of the leading cycle may get too large to be handled efficiently 
or cannot be found at all, see Example~\ref{ex:h1}.
In most cases $h$ cannot  be chosen less than~$0.1$ (as from our numerical tests).
}
\end{remark}

\subsection{Positive systems}
A continuous-time linear switching system
  is called \emph{positive} if every trajectory~$\bx(t)$  
  starting in the positive orthant~$\RR^d_+$  remains  in~$\RR^d_+$  for all~$t$. 
  Positive systems have been studied widely in literature 
  due to many applications. 
  
The positivity of the system is equivalent to that all the matrices~$A_j$ are Metzler,
i.e.\ all off-diagonal entries are nonnegative.
In this case all the matrices~$e^{\,hA_j}, e^{\,mA_j}$
of the discrete-time  system~$\cA_h$ are nonnegative. 
Hence, the matrix~$\Pi$ in the invariant polytope algorithm is also nonnegative, 
and therefore, the Perron-Frobenius theorem implies 
the nonnegativity of the leading eigenvector~$\bx_0$. 
Consequently,  all the sets~$\cV^{(k)}_j$ are nonnegative, i.e.,  
the algorithm runs entirely in~$\RR^d_+$. It is then possible to replace the 
polytopes~$\co_s{\cV_j^{(j)}}$ by the 
positive polytopes~$\co_+{\cV_j^{(j)}} = \{\bx \ge 0 : \bx \le \co_s{\cV_j^{(j)}}\}$, where the inequalities are
understood element wise.

Since the positive polytopes~$Q_j^{(k)}$ are in general much larger than~$P_j^{(k)}$, 
they absorb more points in each iteration.  This reduces significantly the number of vertices of the polytopes and, respectively, the complexity of each iteration.
In fact, this modification of the invariant polytope algorithm  for positive systems works very efficiently even  for large dimensions~$d$.

\section{Numerical results}
\label{sec:num}
We demonstrate the performance of estimate~\eqref{eq.40} from Theorem~\ref{th.20} 
to the computation of the Lyapunov exponent. 
 Given pairs of random matrices and random dwell time $m \in (0,1)$,
Table~\ref{table:metzler1} presents the performance of the Markovian ipa.
For each dimension $d$, we conducted about 15 tests
and we give the median of the best achieved accuracies
i.e.\ the maximal  lower bound from Equation~\eqref{eq.50}.
As test matrices we used $(lhs)$ matrices with normally distributed entries, and
$(rhs)$ Metzler matrices with random integer entries in $[-9, 9]$. To make the examples 
more interesting, we omit simple cases when one matrix dominates others (which often occurs). 
To this end, the matrices got normalized  such that the 2-norm is equal to~$1$.

One can see that the Markovian ipa can compute bounds in reasonable time
(although the needed time depends very strongly on the matrices)
up to dimension $\approx10$ for general matrices, and
up to dimension $\approx500$ for Metzler matrices.

\begin{remark}{\em 
The scripts used to obtain the experimental results can be found at
\href{https://gitlab.com/tommsch/ttoolboxes/-/tree/master/demo/dwelltime}{\emph{gitlab.com\-/tommsch\-/ttool\-boxes/-/tree/master/demo/dwelltime}}.
}
\end{remark}

\begin{table}
\centering
\caption[]{Computation of the Lyapunov exponent for arbitrary systems and for positive systems }
\label{table:metzler1}
\begin{tabular}{rrr}
\multicolumn{3}{c}{Random  matrices} \\ \hline
\multicolumn{1}{c}{dim} &  \multicolumn{1}{c}{$\text{ub}-\text{lb}$} & \multicolumn{1}{c}{time}\Bstrut   \\
\hline
$   2 $ & $  0.016512 $ & $    4s $\Tstrut\\  
$   3 $ & $  0.020846 $ & $   58s $ \\  
$   4 $ & $  0.036887 $ & $  780s $ \\  
$   5 $ & $  0.108899 $ & $ 2300s $ \\  
$   6 $ & $  0.133852 $ & $ 4000s $ \\  
$   7 $ & $  0.234899 $ & $ 3300s $ \\  
$   8 $ & $  0.314734 $ & $ 4900s $ \\  
$   9 $ & $  0.409434 $ & $ 3500s $ \\  
\end{tabular}
~
\begin{tabular}{rrr}
\multicolumn{3}{c}{Random Metzler matrices} \\ \hline
\multicolumn{1}{c}{dim} &  \multicolumn{1}{c}{$\text{ub}-\text{lb}$} & \multicolumn{1}{c}{time}\Bstrut  \\
\hline
$   2 $ & $  0.018851 $ & $   1s $\Tstrut\\
$   5 $ & $  0.008385 $ & $   1s $ \\
$  13 $ & $  0.007760 $ & $   2s $ \\
$  34 $ & $  0.007225 $ & $   4s $ \\
$  89 $ & $  0.009330 $ & $  14s $ \\
$ 144 $ & $  0.006306 $ & $  41s $ \\
$ 233 $ & $  0.005352 $ & $ 150s $ \\
$ 377 $ & $  0.005544 $ & $ 560s $
\end{tabular}
\end{table}

\subsection{Examples}
\label{sec:ex}

\begin{figure}
\centering
\includegraphics[width=.98\linewidth]{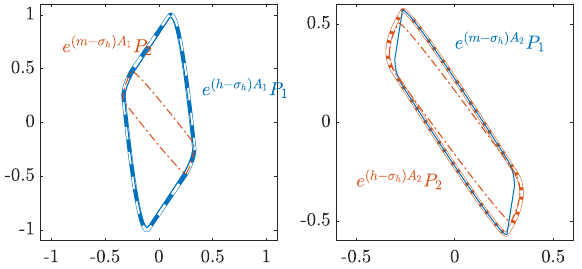}
\caption[]{%
Polytopes from Example~\ref{ex:1}.}
\label{fig:ipaplot}
\end{figure}

We begin with a simple two-dimensional example illustrating 
the Lyapunov exponent computation by the estimates of Theorem~\ref{th.20} 
and the Markovian ipa (Section~5).
\begin{example}\label{ex:1}
Given
dwell time $m = 1$,
two matrices
\begin{equation*}
A_1 =\frac{1}{\sqrt{2} + 2}\begin{bmatrix}0&0\\1&0\end{bmatrix},\quad 
A_2 =\frac{1}{\sqrt{2} + 2}\begin{bmatrix}-2&-2\\-1&-2\end{bmatrix},
\end{equation*}
and discretization step $h=0.2$.
The product 
$\Pi= \allowbreak
(e^{mA_2})^{  1} \allowbreak (e^{hA_2})^{  7} \allowbreak
(e^{mA_1})^{  1} \allowbreak (e^{hA_1})^{180} \allowbreak
(e^{mA_2})^{  1} \allowbreak (e^{hA_2})^{  7} \allowbreak
(e^{mA_1})^{  1} \allowbreak (e^{hA_1})^{181} \allowbreak
(e^{mA_2})^{  1} \allowbreak (e^{hA_2})^{  7} \allowbreak
(e^{mA_1})^{  1} \allowbreak (e^{hA_1})^{183}$
is a leading cycle.
The ipa gives  $\hat\rho(\cA_h)  = \allowbreak  1.0331...$,
or respectively $\sigma_h  = \allowbreak 0.0325...$,
and polytopes $P_1,\, \allowbreak P_2\subseteq \allowbreak \RR^2$ such that
$e^{(m-\sigma_h)A_1}P_2\subseteq \allowbreak P_1$,
$e^{(h-\sigma_h)A_1}P_1\subseteq \allowbreak P_1$,
$e^{(h-\sigma_h)A_2}P_2\subseteq \allowbreak P_2$, and
$e^{(m-\sigma_h)A_2}P_1\subseteq \allowbreak P_2$.

The relevant norms compute to
$\| (A_1 - \sigma_h I)^2 \|_{P_1}  = \allowbreak 0.0083...$, and
$\| (A_2 - \sigma_h I)^2 \|_{P_2}  = \allowbreak 2.8361...$.
Summing up we obtain the bounds
$0.0325 < \allowbreak \sigma(\cA) < \allowbreak 0.0469$.

In Figure~\ref{fig:ipaplot} the polytopes 
$P_1$ (blue-thick-dashed line) and 
$P_2$ (red-thick-dot-dashed line) are plotted,
as well the images under the operators $e^{(m-\sigma_h)A_1}$, \dots (thin-blue line and thin-red-dot-dashed line).
\end{example}

\begin{example}\label{ex:h1}
Let
\begin{equation*}
A_1 \!=\! \begin{bmatrix}
    -1  & -1  & \m1 & -1 \\
    \m1 & -1  & -1  & -1 \\
    \m1 & \m1 & -1  & -1 \\
    \m1 & -1  & \m1 & -1
\end{bmatrix},
A_2 \!=\! \begin{bmatrix}
    -1  & -1  & -1  & -1  \\
    \m1 & -1  & \m1 & \m1 \\
    -1  & \m1 & -1  & -1  \\
    \m1 & -1  & \m1 & \m1
\end{bmatrix},
\end{equation*}
and be the dwell time $m = 0.5$.

In Table~\ref{table:h1} 
we report 
the discretization length \emph{($h$)},
the lower bound $\sigma_h$ of $\sigma(\cA)$ given in~\eqref{eq.40} \emph{(lb)},
the upper bound $\sigma_h - \frac{1}{m} \ln \left(1 - \frac{\|(\cA - \sigma_h I)^2\|}{8} h^2 \right)$ of $\sigma(\cA)$ \emph{(ub)},
and a
\emph{leading cycle}, which is  a periodic switching law 
of the discretized system that produces the fastest growth of trajectories. 
The leading cycle is denoted as~$(a^1_1;\, a^2_1; \ldots;\, a^1_k;\, a^2_k)$, 
where~$a^1_{\vardot}$ and~$a^2_{\vardot}$ are the time of actions of the modes~$A_1$ and~$A_2$ respectively.
For example, the leading cycle $(1.3;\, 1.7)$ in the first line 
of  Table~\ref{table:h1} (the case~$h=0.4$) is the mode~$A_2$ acting the time~$1.3$ 
followed by the time~$1.7$ of the mode~$A_1$.  From the table we see that the leading cycle 
seems to stabilize around $(1.3;\, 1.7)$ as $h$ increases.

The character \emph{``?''} denotes that no leading cycle could be found.
On can see, that for too small values of $h$, the leading cycle cannot be determined,
and thus, the lower bound becomes meaningless.
Nevertheless,
small discretization steps may still give good upper bounds.
\end{example}

\begin{table}
\centering
\caption[]{%
Results for Example~\ref{ex:h1}: the bounds for the Lyapunov exponent and 
the leading cycles depending on~$h$.
}
\label{table:h1}
\begin{tabular}{rrrl}
\multicolumn{1}{c}{$h$} &  \multicolumn{1}{c}{lb}  &         \multicolumn{1}{c}{ub}   & \multicolumn{1}{l}{leading cycle}\Bstrut\\ 
\hline
$ 0.400 $&$   0.0762 $&$  13.2624 $&$ (1.30;\, 1.70)  $\Tstrut\\ 
$ 0.330 $&$   0.0698 $&$   8.6828 $&$ (1.16;\, 1.82;\, 1.49;\, 1.49)  $ \\
$ 0.300 $&$   0.0751 $&$   7.1247 $&$ (1.40;\, 1.70) $ \\
$ 0.250 $&$   0.0742 $&$   4.7571 $&$ (1.25;\, 1.75) $ \\
$ 0.200 $&$   0.0762 $&$   3.0066 $&$ (1.30;\, 1.70) $ \\ 
$ 0.125 $&$   0.0762 $&$   1.1888 $&$ (1.25; \, 1.625)  $ \\ 
$ 0.100 $&$   0.0000 $&$   0.7298 $&$ ? $ \\ 
$ 0.050 $&$   0.0000 $&$   0.3659 $&$ ? $ \\ 
$ 0.040 $&$   0.0000 $&$   0.3796 $&$ ? $ \\ 
$ 0.025 $&$   0.0000 $&$   0.3125 $&$ ? $
\end{tabular}
\end{table}

\appendix

\subsection{Markovian modified Gripenberg algorithm}

The Markovian modified Gripenberg algorithm, is an adaption of the modified Gripenberg algorithm~\cite{M20}.
The main difference is that in each iteration, still, all possible products have to be considered,
but just admissible cycles (w.r.t. to the graph) are used to compute the intermediate bounds.

\begin{algorithm}[Markovian modified Gripenberg Algorithm]\label{alg.10}
The algorithm searches for a leading cycle $\Pi_{\text{max}}$ of a graph $G$ and linear operators $\mathcal{E}$.
\begin{flalign*}
& \rho_{\text{max}} = 0, \quad \Pi_0 = I & \\
& \mathbf{for } \ k = 1,\ldots, K & \\
& \qquad \cE_k = \Pi_{k-1} \times \cE  \quad \text{\footnotesize // all possible products in iteration $k$} & \\
& \qquad \rho_k = \max\{ \rho(E)^{1/T(E)} : E \in \cE_k \text{ and $E$ is a cycle} \}& \\
& \qquad \qquad \text{\footnotesize // where $T(E)$ is the time of action of $E$} & \\
& \qquad \Pi_k = \{ E \in \cE_k :\rho(E)^{1/T(E)} = \rho_k \} & \\
& \qquad \mathbf{if }\ \rho_k > \rho_{\text{max}}\ \mathbf{ then } & \\
& \qquad \qquad  \Pi_{\text{max}} = \Pi_k, \quad \rho_{\text{max}} = \rho_k & \\
& \qquad \mathbf{else} & \\
& \qquad \qquad \Pi_{\text{max}} = \Pi_k \cup \Pi_{\text{max}} & \\
& \qquad \Pi_k = \Pi_k^+ \cup \Pi_k^- = \{ E \in \cE_k : \norm{E}^{1/T(E)} \geq \rho_k^+ \} \ \cup  \\ 
& \qquad \qquad \qquad \{ E \in \cE_k : \norm{M}^{1/T(M)} \leq \rho_k^- \} & \\
& \qquad \qquad \text{\footnotesize // where $\rho_k^\pm \geq \rho_{\text{max}}$ and $\rho_k^\pm$ are such that $\abs{\Pi_k^\pm} = N/2$ } & \\
& \mathbf{return } \ \Pi_{\text{max}}, \rho_{\text{max}}
\end{flalign*}
\end{algorithm}

\begin{example}\label{ex:modgrip}
Given pairs of matrices of various dimensions,
random dwell time $m \in (0,1)$ and random discretization step $h\in(0,m)$.
Table~\ref{table:modgrip} presents the performance of the Markovian modified Gripenberg algorithm \emph{(mg)}
compared to a brute force method \emph{(bf)}.
For each dimension $d$ we conducted 15 tests and we report in how many cases 
the Gripenberg like algorithm, or the brute force algorithm worked better
Furthermore we give the average runtime of the algorithm
(Note though that the Gripenberg like algorithm in general found the final result after approximately $2s$).

As test matrices we used $(a)$ matrices with random normally distributed entries, and
$(b)$ Metzler matrices with random integer entries in $[-9, 9]$,
always normalized such that the \mbox{2-norm} equals~1.

One can see that the Gripenberg like algorithm performs in general faster and better than a brute force method.
\end{example}

\begin{table}
\centering
\caption[]{Results for Example~\ref{ex:modgrip}: Assessment of Markovian modified Gripenberg algorithm}
\label{table:modgrip}
\begin{tabular}{ccccc}
\multicolumn{5}{c}{ $(a)$ Random Gaussian matrices} \\ \hline
$dim$ & mg better  & bf better  & $t_{mg}$   & $t_{bf}$  \Bstrut \\
\hline
$  2  $ & $              33\% $ & $    \phantom{0}7\% $ & $ 2.5s $ & $ 13.6s $\Tstrut\\
$  3  $ & $              20\% $ & $    \phantom{0}0\% $ & $ 2.5s $ & $ 12.7s $\\
$  4  $ & $              33\% $ & $    \phantom{0}7\% $ & $ 2.4s $ & $ 16.4s $\\
$  5  $ & $              13\% $ & $              13\% $ & $ 2.6s $ & $ 16.1s $\\
$  7  $ & $              13\% $ & $    \phantom{0}0\% $ & $ 2.6s $ & $ 12.9s $\\
$  8  $ & $              20\% $ & $    \phantom{0}0\% $ & $ 2.6s $ & $ 15.7s $\\
$ 11  $ & $              20\% $ & $    \phantom{0}0\% $ & $ 2.6s $ & $ 13.3s $\\
$ 13  $ & $    \phantom{0}0\% $ & $    \phantom{0}0\% $ & $ 2.6s $ & $ 15.9s $\\
$ 17  $ & $    \phantom{0}0\% $ & $    \phantom{0}0\% $ & $ 2.6s $ & $ 14.0s $\\
$ 21  $ & $              13\% $ & $    \phantom{0}0\% $ & $ 2.6s $ & $ 17.2s $
\end{tabular}
~\\[1em]\noindent
\begin{tabular}{ccccc}
\multicolumn{5}{c}{$(b)$ Random Metzler matrices} \\ \hline
$dim$ & mg better  & bf better  & $t_{mg}$   & $t_{bf}$ \Bstrut \\
\hline
$  2  $ & $    \phantom{0}0\% $ & $    0\% $ & $ 4.2s $ & $ 27.3s $ \Tstrut\\
$  3  $ & $    \phantom{0}0\% $ & $    0\% $ & $ 2.6s $ & $ 27.5s $\\
$  4  $ & $              40\% $ & $    0\% $ & $ 2.6s $ & $ 18.6s $\\
$  5  $ & $    \phantom{0}0\% $ & $    0\% $ & $ 2.6s $ & $ 26.9s $\\
$  7  $ & $              20\% $ & $    0\% $ & $ 2.6s $ & $ 17.6s $\\
$  8  $ & $              40\% $ & $    0\% $ & $ 2.5s $ & $ 17.2s $\\
$ 11  $ & $              20\% $ & $    0\% $ & $ 2.6s $ & $ 13.5s $\\
$ 13  $ & $    \phantom{0}0\% $ & $    0\% $ & $ 2.6s $ & $ 13.4s $\\
$ 17  $ & $              40\% $ & $    0\% $ & $ 2.5s $ & $ 13.7s $\\
$ 21  $ & $              20\% $ & $    0\% $ & $ 2.6s $ & $ 14.2s $
\end{tabular}
\end{table}
 
\section*{Acknowledgment}

%
%


\section*{References}

\begin{thebibliography}{NN}

\bibitem{B88}
N.~Barabanov, 
\newblock {\em Lyapunov indicators of discrete inclusions i–iii}, 
\newblock Autom. Remote Control, 49 (1988), 152--157.

\bibitem{B08}
C.~Basso, 
\newblock {\em Switch-mode power supplies spice simulations and practical designs}, 
\newblock McGraw-Hill, Inc., New York, NY, USA, 1 edition, 2008.

\bibitem{BM99}
F.~Blanchini and S.~Miani,
\newblock {\em A new class of universal Lyapunov functions for the control of uncertain linear systems},
\newblock IEEE Trans. Automat. Control,   44  (1999) 3, 641--647.

\bibitem{BCM10}
F.~Blanchini, D.~Casagrande and S.~Miani,
\newblock {\em Modal and transition dwell time computation in switching systems: a set-theoretic approach},
\newblock Automatica J. IFAC, 46 (2010) 9, 1477--1482.

\bibitem{BS13}
 C.~Briat and A.~Seuret, 
\newblock {\em Affine characterizations of minimal and mode-dependent dwell times for uncertain linear switched systems},
\newblock IEEE Trans. Automat. Control 58 (2013) 5,   1304--1310.

\bibitem{D14}
X.~Dai,
\newblock {\em Robust periodic stability implies uniform exponential stability of Markovian jump linear systems and random linear ordinary differential equations},
\newblock J. Franklin Inst., 351 (2014), 2910--2937.

\bibitem{CC17}
G.~Chesi and P.~Colaneri.
\newblock {\em Homogeneous rational Lyapunov functions for performance analysis of switched systems with arbitrary switching and dwell time constraints}, 
\newblock IEEE Trans. Automat. Control, 62 (2017) 10, 5124--5137.

\bibitem{CCGMS12}
G.~Chesi, P.~Colaneri, J.~C. Geromel, R.~Middleton, and R.~Shorten.
\newblock {\em A nonconservative LMI condition for stability of switched systems with guaranteed dwell time}, 
\newblock IEEE Trans. Automat. Control, 57 (2012) 5, 1297--1302.

\bibitem{CGPS21} 
Y.~Chitour, N.~Guglielmi, V.\,Yu.~Protasov, and M.~Sigalotti, 
\newblock {\em Switching systems with dwell time: computing the maximal Lyapunov exponent}, 
\newblock Nonlinear Anal. Hybrid Syst. 40 (2021), Paper No. 101021. 

\bibitem{CGP18}
A.~Cicone, N.~Guglielmi, and V.~Y. Protasov, 
\newblock {\em Linear switched dynamical systems on graphs}, 
\newblock Nonlinear Anal. Hybrid Syst., 29 (2018), 165--186.

\bibitem{GC06} 
J.C.~Geromel and P.~Colaneri, 
\newblock {\em Stability and stabilization of continuous-time switched linear systems}, 
\newblock SIAM J. Control Optim. 45 (2006) 5, 1915--1930.

\bibitem{GLP17}
N.~Guglielmi, L.~Laglia, and V.\,Yu.~Protasov, 
\newblock {\em Polytope Lyapunov functions for stable and for stabilizable LSS}.
\newblock Found. Comput. Math., 17 (2017) 2, 567--623.

\bibitem{GP13}
N.~Guglielmi and V.\,Yu.~Protasov, 
\newblock {\em Exact computation of joint spectral characteristics of linear operators}, 
\newblock Found. Comput. Math., 13 (2013), 37--97.

\bibitem{KP23}
V.\,Yu.~Protasov, and R.~Kamalov,
\newblock {\em Stability of Continuous Time Linear Systems with Bounded Switching Intervals},
\newblock SIAM J. Control Optimiz., 61 (2023) 5, 3051--3075.

\bibitem{K14}
V.~Kozyakin, 
\newblock {\em The Berger-Wang formula for the Markovian joint spectral radius}, 
\newblock Linear Alg. Appl., 448 (2014), 315--328.

  \bibitem{K10}
T.~Kr{\"o}ger, 
\newblock {\em On-line trajectory generation in robotic systems: basic concepts for instantaneous reactions to unforeseen (sensor) events}, 
\newblock Springer Berlin Heidelberg, Berlin, Heidelberg, 2010.

\bibitem{L03}
D.~Liberzon, 
\newblock {\em Switching in systems and control}, 
\newblock Systems \& Control: Foundations \& Applications. Birkh\"{a}user Boston, Inc., Boston, MA, 2003.

\bibitem{LM99}
D.~Liberzon and A.~S. Morse, 
\newblock {\em Basic problems in stability and design of switched systems}, 
\newblock IEEE Control Systems Magazine, 19 (1999), 59--70.

\bibitem{M20}
T.~Mejstrik, 
\newblock {\em Improved invariant polytope algorithm and applications}, 
 ACM Trans. Math. Softw., ACM Transactions on Mathematical Software 46 (2020) 3, 1--26. 

\bibitem{MP23}
T.~Mejstrik and V.\,Yu.~Protasov,
\newblock {\em Elliptic polytopes and invariant norms of linear operators},
\newblock Calcolo 60 (2023) 56.

\bibitem{MP89}
A.~Molchanov and Y.~Pyatnitskiy, 
\newblock {\em Criteria of asymptotic stability of differential and difference inclusions encountered in control theory}, 
\newblock Syst. Control Lett., 13 (1989), 59--64.

\bibitem{P19}
P.~Pepe, 
\newblock {\em Converse Lyapunov theorems for discrete-time switching systems with given switches digraphs,}
\newblock IEEE Trans. Autom. Control, 64 (2019) 6, 2502--2508.  

\bibitem{PV09}
A.~Pietrus and V.~Veliov, 
\newblock {\em On the discretization of switched linear systems},
\newblock Syst. Cont. Letters, 58 (2009) 6,  395--399.

\bibitem{PEDJ16}
M.~Philippe, R.~Essick, R.~Dullerud, and R.\,M.~Jungers, 
\newblock {\em Stability of discrete-time switching systems with constrained switching sequences},  
\newblock Automatica, 72 (2016), 242--250. 

\bibitem{PJ15}
M.~Philippe and R.\,M.~Jungers,
\newblock {\em A sufficient condition for the boundedness of matrix products accepted by an automaton},
\newblock Proceedings of the 18th International Conference on Hybrid Systems: Computation and Control, ACM New York, NY, USA (2015), 51--57.

\bibitem{PMJ17}
M.~Philippe, G.~Millerioux, and R.\,M.~Jungers, 
\newblock {\em Deciding the boundedness and dead-beat stability of constrained switching systems}, 
\newblock Nonlinear Anal. Hybr. Syst., 23 (2017), 287--299.

\bibitem{P22}
V.\,Yu.Protasov,  
\newblock {\em Generalized Markov-Bernstein inequalities and stability of dynamical systems}, 
\newblock Proc. Steklov Inst., 319 (2022), 237--252.

\bibitem{PJ16}
V.\,Yu.~Ptotasov and R.\,M.~Jungers, 
\newblock {\em Analysing the stability of linear systems via exponential Chebyshev polynomials,}
\newblock IEEE Trans. Autom. Control 61 (2016) 3, 795--798.

\bibitem{RCS11}
F.~Rossi, P.~Colaneri, and R.~Shorten, 
\newblock {\em Pad\'e discretization for linear systems with polyhedral lyapunov functions}, 
\newblock IEEE Trans. Autom. Control, 56 (2011) 11, 2717--2722. 

\bibitem{SCSS11}
R.~Shorten, M.~Corless, S.~Sajja and S.~Solmaz, 
\newblock {\em On Pad\'e approximations, quadratic stability and discretization of switched linear systems}, 
\newblock Syst.~\& Contr. Letters, 60 (2011) 9,  683--689. 

\bibitem{SFS15}
M.~Souza, A.~Fioravanti, and R.~Shorten, 
\newblock {\em Dwell-time control of continuous-time switched linear systems}, 
\newblock In  Proceedings of the 54th IEEE Conference on Decision and Control (2015),  4661--4666.


\bibitem{SS00}
A.~van~der Schaft and H.~Schumacher, 
\newblock {\em An introduction to hybrid dynamical systems}, 
\newblock vol. 251 of Lecture Notes in Control and Information Sciences, 
Springer-Verlag London, Ltd., London, 2000.


\bibitem{X15}
W.~Xiang,
\newblock {\em On equivalence of two stability criteria for continuous-time switched systems with dwell time constraint},  
\newblock Automatica J. IFAC, 54 (2015), 36--40.

\end{thebibliography}
\end{document}